\documentclass[final]{jotart}
\usepackage{amsmath}
\usepackage{amssymb}
\theoremstyle{proclaim}
\newtheorem{theorem}{Theorem}[section]
\newtheorem{lemma}[theorem]{Lemma}
\newtheorem{corollary}[theorem]{Corollary}
\newtheorem{proposition}[theorem]{Proposition}
\theoremstyle{statement}
\newtheorem{remark}[theorem]{Remark}
\newtheorem{definition}[theorem]{Definition}
\newtheorem{example}[theorem]{Example}
\theoremstyle{fancyproclaim}
\newtheorem*{vartheorem}{ }
\numberwithin{equation}{section}

\usepackage{amsfonts,amsthm,mathrsfs,latexsym,paralist, xypic}
\usepackage{tikz}
\usetikzlibrary{arrows,automata}

\tikzstyle{V}=[fill=black,circle,scale=0.4, outer sep = 4pt]

\DeclareMathOperator{\Ad}{Ad}
\DeclareMathOperator{\Aut}{Aut}

\newcommand{\z}{^{(0)}}
\renewcommand{\2}{^{(2)}}
\newcommand{\inv}{^{-1}}

\newcommand{\bi}{\begin{itemize}}
\newcommand{\ei}{\end{itemize}}
\newcommand{\be}{\begin{enumerate}}
\newcommand{\ee}{\end{enumerate}}

\newcommand{\C}{\mathbb{C}}

\newcommand{\T}{\mathbb{T}}
\renewcommand{\H}{\mathcal{H}}

\newcommand{\G}{\mathcal{G}}
\newcommand{\K}{\mathcal{K}}

\newcommand{\R}{\mathbb{R}}

\newcommand{\Z}{\mathbb{Z}}
\newcommand{\Hom}{\text{Hom}}

\begin{document}
\issueinfo{00}{0}{0000} 
\commby{Kenneth R. Davidson}
\pagespan{101}{110}
\date{Month dd, yyyy}
\revision{Month dd, yyyy}
\title{$K$-theory and Homotopies of 2-cocycles on Transformation Groups}
\author{Elizabeth Gillaspy}
\address{ELIZABETH GILLASPY, Department of Mathematics, University of Colorado - Boulder, Boulder, CO  80309-0395, USA}
\email{elizabeth.gillaspy@colorado.edu}


\begin{abstract}
This paper constitutes a first step in the author's program to investigate the question of when a homotopy of 2-cocycles $\omega = \{\omega_t\}_{t \in [0,1]}$ on a locally compact Hausdorff groupoid $\mathcal{G}$ induces an isomorphism of the $K$-theory groups of the reduced twisted groupoid $C^*$-algebras:
\[K_*(C^*_r(\mathcal{G}, \omega_0)) \cong K_*(C^*_r(\mathcal{G}, \omega_1)).\]
Generalizing work of Echterhoff et al.~from \cite{ELPW}, we show that if $\G = G \ltimes X$ is a transformation group such that $G$ satisfies the Baum-Connes conjecture with coefficients, a homotopy $\omega = \{\omega_t\}_{t \in [0,1]}$ of 2-cocycles on $G \ltimes X$ gives rise to an isomorphism 
\[K_*(C^*_r(G \ltimes X, \omega_0)) \cong K_*(C^*_r(G \ltimes X, \omega_1)).\]
\end{abstract}
\maketitle

\section{INTRODUCTION}
In this paper, we study the question of when a homotopy of 2-cocycles $\omega = \{\omega_t\}_{t \in [0,1]}$ on a transformation group $G \ltimes X$ induces an isomorphism of the twisted $K$-theory groups
\[K_*(C^*_r(G \ltimes X, \omega_0)) \cong K_*(C^*_r(G \ltimes X, \omega_1)).\]
Variants on this question have been investigated by Packer and Raeburn in \cite{twisted-xprod1}, \cite{twisted-xprod2} and Echterhoff, L\"uck, Phillips, and Walters in \cite{ELPW}, but its origins lie in the results established about the  structure and $K$-theory of the rotation algebras by Rieffel, Pimsner and Voiculescu, among others, in the early 1980s.

For $\theta \in \R$, let $A_\theta$ denote the universal rotation algebra; that is, the algebra generated by two unitaries $u,v$ satisfying the commutation relation 
\[ uv = vu \mathrm{e}^{2\pi \mathrm{i} \theta}.\]
We can also realize $A_\theta$ as a twisted group $C^*$-algebra: 
\[A_\theta = C^*_r(\Z^2, c_\theta),\]
 where the twisting 2-cocycle $c_\theta: \Z^2 \times \Z^2 \to \T$ is given by 
\[ c_\theta\left( (j,k), (m,n) \right) = \mathrm{e}^{ 2\pi \mathrm{i} \theta km}.\]
Note that the map $\theta \mapsto c_\theta\left( (j,k), (m,n) \right) $ is continuous for each $( (j,k), (m,n)) \in \Z^2 \times \Z^2$; because of this continuity, we say that $\{c_\theta\}_{\theta \in \R}$ is a \emph{homotopy of 2-cocycles} on $\Z^2$.

Rieffel proved in 1981 \cite{irr-k-thy} that although 
\[A_\theta \cong A_{\theta'} \Leftrightarrow [\theta] = \pm [\theta'] \in \R/\Z,\]  the Morita equivalence relation on the rotation algebras is less strict: two rotation algebras $A_\theta, A_{\theta'}$ are Morita equivalent iff $\theta, \theta'$ are in the same orbit of a natural action of $GL_2(\Z)$ on $\R$ (which is defined in Section 2 of \cite{irr-k-thy}). 
Finally, Pimsner and Voiculescu proved in \cite{pims-voic} that for any $\theta \in \R$, we have 
\begin{equation}
K_0(A_\theta) \cong \Z \oplus \Z \cong K_1(A_\theta).\label{kthy-irr}
\end{equation}

 The  example of the rotation algebras thus tells us that $K$-theory is the appropriate $C^*$-algebraic invariant for which to ask whether a homotopy of 2-cocycles is detected by this invariant: the homotopy of 2-cocycles that gives rise to the rotation algebras induces a $K$-theoretic equivalence, but not a $C^*$-algebraic isomorphism or even a Morita equivalence.    In other words, these more precise invariants 
will not, in general, be preserved under a homotopy of 2-cocycles.

The Pimsner-Voiculescu result \eqref{kthy-irr} about the $K$-theory of the rotation algebras was vastly generalized in a 2010 paper \cite{ELPW} by Echterhoff, L\"uck, Phillips, and Walters.  For a second countable locally compact Hausdorff group $G$  that satisfies the Baum-Connes conjecture with coefficients $\K$, Echterhoff et al.~proved in Theorem 1.9 of \cite{ELPW} that if $\{\omega_t\}_{t\in [0,1]}$ is a homotopy of 2-cocycles on $G$, then 
\[K_*(C^*_r(G, \omega_0)) \cong K_*(C^*_r(G, \omega_1)).\]
Our main theorem in this paper is an  extension of this result  to the case when $G$ is a transformation group $G = H \ltimes X$:
\begin{vartheorem}[Theorem \ref{main}]
Let $G$ be a second countable locally compact Hausdorff group acting on a second countable locallly compact Hausdorff space $X$ such that $G$ satisfies the Baum-Connes conjecture with coefficients, and let $\omega$ be a homotopy of continuous 2-cocycles on the transformation group $G \ltimes X$.  For any $t \in [0,1]$,  the $*$-homomorphism \[q_t: C^*_r(G \ltimes X \times [0,1], \omega) \to C^*_r(G \ltimes X, \omega_t),\]
given on $C_c(G \ltimes X \times [0,1])$ by evaluation at $t \in [0,1]$, induces an isomorphism 
\[K_*(C^*_r(G \ltimes X \times [0,1], \omega)) \cong K_*(C^*_r(G \ltimes X, \omega_t)).\] 
\end{vartheorem}

\begin{remark}
Higson and Kasparov proved in \cite{higson-kasparov} that every a-T-menable group satisfies the Baum-Connes conjecture with coefficients.  Examples of such groups include all amenable groups (hence all compact, abelian, and solvable groups), all free groups, and the Lie groups $SO(n, 1)$ and $SU(n,1)$.  Moreover, a-T-menability is inherited by closed subgroups, so any closed subgroup of the above groups also satisfies the Baum-Connes conjecture with coefficients.
\label{b-c}
\end{remark}

\begin{remark}
Theorem \ref{main} can also be viewed as a generalization of Theorem 4.2 from the 1990 paper \cite{twisted-xprod2} of Packer and Raeburn.
  When translated into the notation of the current paper, Packer and Raeburn's Theorem 4.2  states that if $G$ is a discrete subgroup of a solvable simply-connected Lie group,
 and $G$ acts on a locally compact Hausdorff space $X$, then a homotopy of 2-cocycles $\{\omega_t\}_{t \in [0,1]}$ on $G \ltimes X$ induces an isomorphism
\[ K_*(C^*_r(G \ltimes X, \omega_0)) \cong K_*(C^*_r(G \ltimes X, \omega_1)).\]
By Theorem 8.2 of \cite{julg-kasparov}, discrete subgroups of solvable simply-connected Lie\\ groups satisfy the Baum-Connes conjecture with coefficients, 
so our Theorem \ref{main} applies to a much broader class of groups than those covered in Theorem 4.2 of \cite{twisted-xprod2}.
\end{remark}

\subsection{Context and Future work}
A transformation group $G \ltimes X$ is an example of a \emph{groupoid}, a class of mathematical objects that includes groups, group actions, equivalence relations, and group bundles.  The study of the full and reduced $C^*$-algebras $C^*(\G), C_r^*(\G)$ associated to a locally compact groupoid $\G$ was initiated by Jean Renault in \cite{renault}, and has been pursued actively by many researchers.
  Although Renault also defined the twisted groupoid $C^*$-algebras $C^*(\G, \omega), C^*_r(\G, \omega)$ for a 2-cocycle $\omega \in Z^2(\G, \T)$ in \cite{renault}, these objects have received relatively little attention until recently. However, it has now become clear that twisted groupoid $C^*$-algebras can help answer many questions about  the structure of untwisted groupoid $C^*$-algebras (cf.~\cite{cts-trace-gpoid-II}, \cite{cts-trace-gpoid-III}, \cite{clark-aH}, \cite{dd-fell}, \cite{jon-astrid}), as well as classifying those $C^*$-algebras which admit diagonal subalgebras (also known as Cartan subalgebras) --- cf.~\cite{c*-diagonals}. 
    In another direction,  \cite{TXLG} 
  establishes how the $K$-theory of twisted groupoid $C^*$-algebras connects to the classification of $D$-brane charges in string theory.
  

Motivated by these applications of twisted groupoid $C^*$-algebras and their $K$-theory, we propose to investigate the question of when a homotopy of the twisting 2-cocycle gives rise to a $K$-theoretic isomorphism 
\begin{equation}
K_*(C^*_{r}(\G, \omega_0)) \cong K_*(C^*_{r}(\G, \omega_1)).
\label{k-equiv}
\end{equation}
In addition to the transformation-group case $\G = G \ltimes X$ considered in the present paper, we will address this question in a forthcoming paper \cite{eag-kgraph} for the case when $\G = \G_\Lambda$ is the groupoid  associated to a higher-rank graph $\Lambda$.  (A generalization of directed graphs, higher-rank graphs and their associated groupoids and $C^*$-algebras were introduced in \cite{kp}.) 

We also explore in \cite{eag-bdl} the situation where $\G$ is a locally compact group bundle $\pi: \G \to M$ over a paracompact space.  Our interest in these groupoids was inspired by the special case when $\G$ is a symplectic vector bundle; in this case, Theorem 1 of \cite{plymen-weyl-bdl} combines with the Thom isomorphism to tell us that the homotopy of 2-cocycles associated to the symplectic form gives rise to the isomorphism \eqref{k-equiv} of the $K$-theory groups of the twisted groupoid $C^*$-algebras.  Corollary 3.4 of \cite{eag-bdl} presents a substantial generalization of this result.

We are unaware of any examples of groupoids $\G$ and homotopies $\omega$ of cocycles where \eqref{k-equiv} fails to hold. 

\subsection{Outline and Standing Hypotheses}
To tackle the proof of Theorem \ref{main}, we will first need to understand  the reduced twisted groupoid $C^*$-algebra $C^*_r(G \ltimes X, \omega)$; Section \ref{2} reviews this construction for general groupoids  $\G$, in order to establish the context and simplify formulas.  In Section \ref{4} we show that a homotopy of 2-cocycles on a compact groupoid $\G$ gives rise to a trivial bundle of $C^*$-algebras over $[0,1]$.  Section \ref{pr}  describes carefully the isomorphism between the   twisted crossed product $C^*$-algebra defined by Packer and Raeburn in \cite{twisted-xprod1}, and the twisted groupoid $C^*$-algebra described in Section \ref{2}; this material is no doubt well known to experts but we were unable to find a precise reference, so we include it here.  Finally, in Section \ref{6} we combine the results of the preceding sections in order to prove Theorem \ref{main}, following the line of argument presented in Theorem 1.9 of \cite{ELPW}.  Since this argument relies on a technical result in equivariant $KK$-theory (namely, Theorem 1.5 of \cite{going-down-functors}), Section \ref{6} opens with a section reviewing those elements of Kasparov's equivariant $KK$-theory that we will need to invoke.

In order to use the $KK$-theoretic result from \cite{going-down-functors}, we will need to assume that all $C^*$-algebras under consideration are separable and all groups are second countable.  The same hypotheses are invoked in \cite{twisted-xprod1}.  Consequently, in Section \ref{pr} we begin to require all groups and spaces in question to be Hausdorff, locally compact and second countable; for the material earlier in the paper we do not need to assume second countability.  However, throughout this paper we will assume that all groups and spaces are locally compact Hausdorff.

\section{GROUPOIDS}
\label{2}

In this section we review many of the basic concepts and constructions from the theory of groupoid $C^*$-algebras, focusing on the case where our groupoid $\G$ is a transformation group $G \ltimes X$.  Many of the definitions are given first for general groupoids, to simplify notation and to contextualize our work in this paper.

\begin{definition} A \emph{groupoid} is a set $\mathcal{G}$ equipped with a subset $\mathcal{G}^{(2)} \subseteq \mathcal{G} \times \mathcal{G}$ (called the \emph{set of composable pairs}), a multiplication map $\mathcal{G}^{(2)} \to \G$ given by $(x, y) \mapsto xy$, and an inverse map $\mathcal{G} \to \mathcal{G}$, written $x \mapsto x^{-1}$, such that:
\bi
\item If $(x, y), (y,z) \in \mathcal{G}^{(2)}$, then so are $(x, yz)$ and $(xy, z)$, and $x(yz) = (xy)z$;
\item For any $x \in \mathcal{G}$, the pair $(x, x^{-1}) \in \mathcal{G}^{(2)}$;
\item For any $(x, y) \in \mathcal{G}^{(2)}$, we have $x^{-1}(xy) = y$ and $(xy)y^{-1} = x$.
\ei
Any groupoid comes equipped with \emph{range} and \emph{source} maps $r, s: \mathcal{G} \to \mathcal{G}$, defined by $r(x) = x^{-1}x$ and $s(x) = xx^{-1}$. 
 Observe that $r$ and $s$ have a common image, which we call the \emph{unit space} of $\mathcal{G}$, and denote $\mathcal{G}^{(0)}$.
 If $u \in \mathcal{G}\z$, we will write 
 \[\mathcal{G}_u = \{ x \in \mathcal{G}: s(x) = u\}; \qquad \mathcal{G}^u = \{ x \in \mathcal{G}: r(x) = u\}.\]
\end{definition}

In this paper we will consider almost exclusively the following class of groupoids. 


\begin{example}
Suppose $G$ is a group acting (on the left) on a space $X$.  We define a groupoid $G \ltimes X$, called the \emph{transformation group}, to be $G \times X$ as a set, with $(G\ltimes X)\z = X$:
\[s(\gamma,u) = \gamma\inv \cdot u, \quad r(\gamma,u) = u.\]
In other words, we can think of an element $x = (\gamma, u) \in G \ltimes X$ as an arrow (labeled $\gamma$) taking us from the point $v := \gamma\inv \cdot u \in X$ to the point $u$:
\[
          		\begin{tikzpicture}
			\node (Left) [V, label=left:{$u$}] at (-1,0){};
			\node (Right) [V, label=right:{$v$}] at (1,0){};
		\begin{scope}[->,  line width=1]
			\draw  (Right) to[bend right] node[above] {$\gamma$}  (Left); 	
		\end{scope}
	\end{tikzpicture}
\]

Then $\left((\gamma, u), (\eta, v)\right) \in (G \ltimes X)\2 \Leftrightarrow \gamma\inv \cdot u =v,$
and 
\[(\gamma,u) \cdot (\eta, \gamma\inv \cdot u) = (\gamma \eta, u); \qquad (\gamma,u)\inv = (\gamma\inv, \gamma\inv \cdot u).\]
Note that for any $u \in X$, we have $(G \ltimes X)_u \cong (G \ltimes X)^u \cong G$.
\label{transf-gp}
\end{example}

\begin{example}
Let $G$ be a group acting trivially on the one-point space. A moment's thought reveals that the associated transformation group is isomorphic to the group $G$.  In other words, groups are examples of transformation groups and hence of groupoids.
\end{example}

\begin{remark}
In this paper, we will use $x,y,z$ to denote elements of an arbitrary groupoid $\G$, whereas Greek letters such as $\gamma, \eta$ will denote elements of a group $G$.  The letters $u,v$ will denote elements of the unit space of our groupoid -- so $u, v \in X$ if the groupoid $\G$ under consideration is a transformation group $G \ltimes X$.
\end{remark}

\begin{definition} We say that a groupoid $\mathcal{G}$ is a \emph{locally compact Hausdorff groupoid} or \emph{LCH groupoid} if $\mathcal{G}$ is a locally compact Hausdorff  topological space such that multiplication and inversion are continuous (when $\mathcal{G}^{(2)}$ has the topology induced by the product topology on $\mathcal{G} \times \mathcal{G}$).
\end{definition}


\begin{remark}
Since the range and source maps can be constructed from multiplication and inversion, it follows that $r,s$ are also continuous in a LCH groupoid.
\end{remark}

\begin{example} If $\mathcal{G} = G \ltimes X$ is a transformation group, where $G$ and $X$ are locally compact Hausdorff spaces and the action of $G$ on $X$ is continuous, then the product topology on $G \times X$ makes $G \ltimes X$ into a LCH groupoid.  We will always use this topology on $G \ltimes X$ in this paper.
\label{top}
\end{example}

\begin{definition} Let $\mathcal{G}$ be a LCH groupoid.  A (continuous) map $\omega: \mathcal{G}^{(2)} \to \mathbb{T}$ is called a \emph{(continuous) 2-cocycle} if 
\begin{equation} \omega(x, y) \omega(xy, z)= \omega(x, yz) \omega(y, z)
\label{cocycle}
\end{equation}
whenever $(x, y), (y, z) \in \mathcal{G}^{(2)}$, and if 
\begin{equation} \omega(x, s(x)) = 1 = \omega(r(x), x)\label{units} \end{equation}
for any $x \in \mathcal{G}$.
\end{definition}

As 2-cocycles are the only flavor of cocycle we will discuss in this paper, we will usually drop the 2 and  refer to them simply as \emph{cocycles}.

\begin{example}
For any groupoid $\G$, the function $\omega: \G^2 \to \T$ given by $\omega(x,y) =1 \ \forall (x,y) \in \G\2$ is a 2-cocycle, called the \emph{trivial 2-cocycle}.
\end{example}

\begin{example}
Let $\G = \Z^2$ and fix $t \in \R$.  Then the function $c_t: \Z^2 \times \Z^2 \to \T$ given by 
\[c_t\left( (j,k), (m,n) \right) := \mathrm{e}^{2\pi \mathrm{i} t (km)}\]
is a 2-cocycle.
\end{example}

\begin{remark}
\label{coh-gp}
If $\omega_0, \omega_1$ are two 2-cocycles on $\G$, then the pointwise product $\omega_0 \omega_1$ also satisfies the cocycle condition, as does the pointwise inverse $\omega_0\inv$ for any cocycle $\omega_0$. Thus, the set of 2-cocycles on $\G$ forms a group, usually denoted $Z^2(\G, \T)$.
\end{remark}

\begin{definition}
Let $\mathcal{G}$ be a LCH groupoid.  We consider $\mathcal{G} \times [0,1]$ to be a LCH groupoid by equipping it with the product topology (using the standard topology on $[0,1]$).  Set 
\[\left(\mathcal{G}\times[0,1]\right)\2 = \mathcal{G}\2 \times [0,1].\]
That is, the groupoid structure on $\G \times [0,1]$ is that of a bundle of groupoids over $[0,1]$, so that all groupoid operations on $\G \times [0,1]$ preserve the fibers.

  We say that $\omega$ is a \emph{homotopy of 2-cocycles on $\mathcal{G}$} if $\omega: \left(\mathcal{G} \times [0,1]\right)\2 \to \T$ is a continuous 2-cocycle.
  \end{definition}

\begin{remark}
Every homotopy of cocycles on $\G$ gives rise to a  family \\ $\{\omega_t\}_{t \in [0,1]}$ of continuous 2-cocycles on $\G$ which varies continuously in $t$.
\end{remark} 

In order to associate $C^*$-algebras to LCH groupoids $\G$ and continuous cocycles $\omega$, we will start by turning the continuous compactly supported functions $C_c(\G)$ into a convolution algebra, and then taking an appropriate completion.  To do so, we need to integrate over $\G$, which requires a Haar system on $\G$.

\begin{definition}
Let $\G$ be a locally compact Hausdorff groupoid.  A collection $\{\lambda^u\}_{u \in \G\z}$ of non-negative Radon measures is a \emph{Haar system} if 
\bi
\item supp($\lambda^u) = \G^u$ for any $u \in \G\z$,
\item For any $f \in C_c(\G)$, the function $u \mapsto \int_{\G^u} f(x) \mathrm{d}\lambda^u(x)$ is in $C_0(\G\z)$,
\item The system of measures is left-invariant: For any $f \in C_c(\G)$ and any $x \in \G$,
\[\int_{\G^{s(x)}} f(xy) \mathrm{d}\lambda^{s(x)}(y) = \int_{\G^{r(x)}} f(y) \mathrm{d}\lambda^{r(x)}(y).\]
\ei
\end{definition}

\begin{example}
If $G$ is a LCH group, then $G\z $ contains only one point, namely, the unit $e$ of $G$, and so the first two conditions of the definition are irrelevant.  The third condition tells us that any Haar measure on $G$ constitutes a Haar system.
\end{example}

Unlike for groups, Haar systems for groupoids need not exist or be unique.  One generally assumes from the beginning the existence of a fixed Haar system on the groupoid in question, a precedent we will follow in this paper.

\begin{example} If $\G = G \ltimes X$ for a LCH group $G$ and a LCH space $X$, then fix a Haar measure $\lambda$ on $G$.  Setting $\lambda^u = \lambda$  for every $u \in X$ makes $\{\lambda^u\}_{u \in X}$ into a Haar system on $G \ltimes X$.  We will always use this Haar system on a transformation group in this paper.
\end{example}

Given a continuous cocycle $\omega$ on a LCH groupoid $\G$, and a Haar system $\{\lambda^u\}_{u \in \G\z}$ on $\G$,  we can turn $C_c(\G)$ into a convolution algebra $C_c(\G, \omega)$:
\begin{align}
 f *_\omega g(x) & := \int_{\G^{r(x)}} f(y) g(y\inv x)\, \omega\left( y, y\inv x \right) \, \mathrm{d}\lambda^{r(x)}(y)\\
f^*(x) & := \overline{f(x\inv)\, \omega\left(x, x\inv \right)}.
\label{algebra-str}
\end{align}
In the case when $\G = G \ltimes X$ is a transformation group, the formula for convolution becomes 
\begin{equation}
f*_\omega g(\gamma, u) := \int_G f(\eta, u) g(\eta\inv \gamma, \eta\inv  u)\, \omega\left( (\eta, u), (\eta\inv \gamma, \eta\inv u)\right) \, \mathrm{d}\lambda(\eta).
\label{conv}
\end{equation}

Convolution multiplication is evidently linear, but also is easily checked to be associative and to satisfy $f^* *_\omega g^* = (g *_\omega f)^*$ (cf.~\cite{renault} II.1).  One needs to invoke the cocycle condition \eqref{cocycle} in order to show associativity.

To make $C_c(\G, \omega)$  into a $C^*$-algebra, we need to complete it with respect to a suitable norm.  There are two $C^*$-algebras canonically associated to $C_c(\G, \omega)$; however, in this paper we will only be concerned with the reduced twisted $C^*$-algebra of a transformation group $\G = G \ltimes X$, so we present only this construction here.  The reader who wishes to understand groupoid $C^*$-algebras in more generality is referred to \cite{muhly-gpoids} or \cite{renault}.

\begin{definition}
\label{l2}
Given a groupoid $\G$, let $\mu$ be a measure on $\G\z$ with full support.  As in \cite{muhly-gpoids} Definition 2.45, we construct a measure $\nu\inv$ on $\G$, which is the pullback, under the map $x \mapsto x\inv$, of the measure induced by $\mu$ on $\G$.  To be precise, a function $\xi$ on $\G$ is in $L^2(\G, \nu\inv)$ iff 
\[\| \xi\|_2^2 := \int_{\G\z} \int_{\G^u} | \xi(x\inv) |^2\mathrm{d}\lambda^u(x) \mathrm{d}\mu(u) < \infty.\]
In the case when $\G = G \ltimes X$, so $\G\z = X$ and $\lambda^u$ is Haar measure on $G$, the formula above becomes
\[ \|\xi\|_2^2 = \int_X \int_G |\xi(\gamma\inv, \gamma\inv \cdot u)|^2 \, \mathrm{d}\lambda(\gamma) \mathrm{d}\mu(u).\]
\end{definition}

Convolution multiplication defines a $*$-representation  of $C_c(\G, \omega)$ on\\  $L^2(\G, \nu\inv)$ which we will denote $\text{Ind } \mu$: 

\begin{definition}
For $f \in C_c(\G, \omega)$, we define the operator $\text{Ind}\, \mu(f)$ on $L^2(\G, \nu\inv)$ by
\begin{align*}
\text{Ind }\mu(f) \xi(x) := (f *_\omega \xi)(x) = \int_{\G} f(xy) \xi(y\inv )\, \omega(xy, y\inv ) \, \mathrm{d}\lambda^{s(x)}(y)
\end{align*}
for $\xi \in L^2(\G, \nu\inv)$.
\end{definition}

\begin{remark}
The notation $\text{Ind}\,\mu$ indicates that this representation of\\ $C_c(\G, \omega)$ is induced from the representation  of $C_0(\G\z)$ on $L^2(\G\z, \mu)$ by multiplication operators.
\end{remark}

  A variation on Theorem 6.18 from \cite{folland} 
tells us that $\text{Ind }\mu(f) \xi \in L^2(\G, \nu\inv)$ whenever $f \in C_c(\G, \omega)$ and $\xi \in L^2(\G, \nu\inv)$, and moreover that the operator norm $\| \text{Ind }\mu(f)\| $ is bounded by the $L^1$-norm $\|f\|_{\nu\inv,1}$ of $f$ with respect to the measure $\nu\inv$. Since convolution multiplication is linear, associative, and $*$-preserving, it follows that $\text{Ind}\, \mu$ is a $*$-representation of $C_c(\G, \omega)$.

\begin{definition}
The \emph{reduced norm $\|\cdot \|_r$} on $C_c(\G, \omega)$ is given by
\begin{align*}
\|f\|_r & := \|\text{Ind } \mu(f)\| .
\end{align*}
The completion of $C_c(\G, \omega)$ with respect to the norm $\| \cdot \|_r$ is the \emph{reduced twisted groupoid $C^*$-algebra $C^*_r(\G, \omega)$}.
\label{reduced}
\end{definition}

\begin{remark}
One can check that any full measure $\mu$ on $\G\z$ will give rise to an equivalent norm to $\| \cdot \|_r$, so the $C^*$-algebra $C^*_r(\G, \omega)$ is independent of $\mu$.  In fact, there are multiple equivalent definitions of the reduced norm (cf.~\cite{moore-schochet} Definition 6.3 and \cite{renault} Definition III.2.8). 
\end{remark}

\subsection{Examples}
We now present a few examples of twisted transformation-\\ group $C^*$-algebras.
\begin{example}
\label{twisted-gp}
If $X = pt$ is trivial, then any 2-cocycle $\omega$ on $G \ltimes X$ is simply a 2-cocycle on the group $G$, and the reduced twisted transformation-group $C^*$-algebra is  the reduced twisted group $C^*$-algebra $C^*_r(G, \omega)$.
\end{example}

\begin{example} (This example is a special case of Proposition \ref{pr&gpoid} below.)
Given a nontrivial action $\alpha$ of $G$ on $X$, consider the trivial 2-cocycle $1$ on $G \ltimes X$.  In this case, $C^*_r(G \ltimes X, 1)$ is isomorphic to the reduced crossed product $C^*$-algebra $C_0(X) \rtimes_r G$ (cf. Definition 7.7 of \cite{xprod}), where the action of $G$ on $C_0(X)$ is induced from $\alpha$.  The isomorphism $\phi: C^*_r(G \ltimes X, 1) \to C_0(X) \rtimes_r G$ is given on the dense subalgebra  $C_c(G \times X)$  by
\[\phi(f)(\gamma, u) = \Delta(\gamma)^{-1/2}f(\gamma, u).\]
Here $\Delta$ denotes the modular function of the group $G$.

\end{example}

\begin{example}
\label{nct}
Let $G \ltimes X$ be a transformation group, and let $\omega: G\times G \to \T$ be a continuous 2-cocycle on the group $G$.  
Then $\omega$ also defines a 2-cocycle $\tilde{\omega}$ on $G \ltimes X$:
\[\tilde{\omega}\left( (\gamma, u), (\eta, \gamma\inv u) \right) := \omega(\gamma, \eta).\]
Although $\tilde{\omega}$ does not depend on $X$, the associated twisted groupoid $C^*$-algebra $C^*_r(G \ltimes X, \tilde{\omega})$ is not, in general, isomorphic to $C_0(X) \otimes C^*_r(G, \omega)$.  For example, consider the case $X = \T^2, G = \Z^2$, where the action of $\Z^2$ on $\T^2$ is given by 
\[(m,n) \cdot (z, w) = (z a_1^n, w a_2^m)\]
for a fixed $(a_1, a_2) \in \T^2$, and where the 2-cocycle $\omega$ on $\Z^2$ is given by 
\[\omega\left( (m,n), (j,k) \right) = a_3^{nj}\]
for $a_3 \in \T$ fixed.  Then the twisted transformation group $C^*$-algebra $C^*_r(\Z^2 \ltimes \T^2, \tilde{\omega})$ is not isomorphic to $C(\T^2) \otimes C^*_r(\Z^2, \omega)$.

To see this, let $\delta$ denote the Kronecker delta function.  Using the formulas given in \eqref{algebra-str} for the adjoint and product in $C_c(\Z^2 \ltimes \T^2, \tilde{\omega})$, one can check that the functions 
\begin{align*}
 \phi_1((m,n), (z,w)) = \delta_{(1,0)}(m,n) & \qquad \phi_2((m,n), (z,w)) = \delta_{(0,1)}(m,n) \\
 \psi_1((m,n), (z,w)) = z \delta_{(0,0)}(m,n) & \qquad \psi_2((m,n),(z,w)) = w \delta_{(0,0)}(m,n)
\end{align*}
are unitaries which generate the $*$-algebra $C_c(\Z^2 \ltimes \T^2, \tilde{\omega})$, and which satisfy the commutation relations 
\begin{align}
\phi_1 \phi_2 = a_3 \phi_2\phi_1 & \qquad \phi_i \psi_j = a_j \psi_j \phi_i \textrm{ if } i\not= j\\
\phi_i \psi_i = \psi_i \phi_i & \qquad  \psi_i \psi_j = \psi_j \psi_i \ \forall \ i, j \in \{1,2\}.
\label{nct-comm-relns}
\end{align}

While $C(\T^2) \otimes C^*_r(\Z^2, \omega)$ is also generated by four unitaries, two of those (corresponding to the generators of $C(\T^2))$ are central.  Since none of the generators $\phi_1, \phi_2, \psi_1, \psi_2$ is central (unless $a_1 = a_2 = 1$), it follows that $C^*_r(\Z^2 \ltimes \T^2, \tilde{\omega})$ is only isomorphic to $C(\T^1) \otimes C^*_r(\Z^2, \omega)$ when $\Z^2$ acts trivially on $\T^2$.

In fact, $C^*_r(\Z^2 \ltimes \T^2, \tilde{\omega})$ can be realized as a twisted group $C^*$-algebra\\ $C^*_r(\Z^4, \sigma)$: writing $a_j = \mathrm{e}^{2\pi \mathrm{i} \theta_j}$, the 2-cocycle $\sigma$ is given by
\[\sigma( r, s) = \mathrm{e}^{\pi \mathrm{i} r \cdot \Theta s}\]
for $r,s \in \Z^4$, where 
\[\Theta = \begin{bmatrix}
0 & \theta_3 & 0& \theta_2 \\
-\theta_3 & 0 & \theta_1 & 0 \\
0 & -\theta_1 & 0 & 0 \\
-\theta_2 & 0 & 0 & 0
\end{bmatrix}.\]
The isomorphism $C^*_r(\Z^2 \ltimes \T^2, \tilde{\omega}) \cong C^*_r(\Z^4, \sigma)$ is given on the unitary generators by 
\[\phi_1 \mapsto \delta_{e_1}, \quad \phi_2 \mapsto \delta_{e_2}, \quad \psi_1 \mapsto \delta_{e_3}, \quad \psi_2 \mapsto \delta_{e_4}.\]
Thus, $C^*_r(\Z^2 \ltimes \T^2, \tilde{\omega})$ is an example of a noncommutative 4-torus.
\end{example}

\begin{example}[cf. Example 1.5 in \cite{selective-survey}, Theorem 4.1 in \cite{twisted-xprod1}]
Generalizing Example \ref{nct}, let $G$ be a locally compact second countable group with closed abelian normal subgroup $N$.  Then $G/N$ acts on $\hat{N}$ by 
\[\overline{\gamma}\cdot \phi(n) = \phi(\gamma\inv n \gamma);\]
the action is well-defined because $N$ is abelian.  Similarly, by choosing a section $s$ of the short exact sequence 
\[1 \to N \to G \to G/N \to 1\]
we can define a 2-cocycle $\sigma$ on the transformation group $G/N \ltimes \hat{N}$:
\[\sigma( (\overline{\gamma}, \phi), (\overline{\eta}, \overline{\gamma}\inv \cdot \phi) ) = \phi\left(s(\overline{\gamma})s(\overline{\eta}) s(\overline{\gamma\eta})\inv\right).\]

If we stop here, then the Fourier transform on $\hat{N}$ gives us an isomorphism $\psi: C^*_r(G/N \ltimes \hat{N}, \sigma) \to C^*_r(G)$; this follows from Proposition \ref{pr&gpoid} below and Theorem 4.1 of \cite{twisted-xprod1}. 
However, if we have an action of $G/N$ on another locally compact Hausdorff space $X$, then $\sigma$ gives us a 2-cocycle $\tilde{\sigma}$ on $G/N \ltimes \hat{N} \times X$ which is independent of $X$, and the twisted transformation-group $C^*$-algebra $C^*_r(G/N \ltimes \hat{N} \times X, \tilde{\sigma})$ is not, in general, isomorphic to a tensor product $C_0(X) \otimes C^*_r(G/N \ltimes \hat{N}, \sigma)$ or even to a crossed product $C_0(X \times \hat{N}) \rtimes_r G/N$.  
\end{example}

\section{THE COMPACT CASE}
\label{4}
As mentioned in the Introduction, our proof of Theorem \ref{main} was inspired by a 2010 paper \cite{ELPW} by Echterhoff, L\"uck, Phillips, and Walters.  In particular, Theorem 1.9 in that paper establishes a version of our Theorem \ref{main} for groups, rather than transformation groups.
The idea of the proof of Theorem 1.9 in \cite{ELPW} is to use a technical theorem from \cite{going-down-functors} to reduce the question to the case of $G$ a compact group, which is much more tractable thanks to results of Echterhoff and Williams in \cite{local-inner-C_0} 
on $C_0(X)$-linear actions on continuous trace algebras.  The same  technique of reduction to the compact case works in the case of a transformation group as well, as we shall see, so we will begin by examining the case when our transformation group is compact.

The main result in this section is the following:

\begin{proposition}
Let $\omega$ be a  homotopy of  cocycles on a compact Hausdorff\\ groupoid $\G $.  Then the  map 
\[q_t: C^*_r(\G \times [0,1], \omega) \to C^*_r(\G, \omega_t)\]
 given on $C_c(\G \times [0,1])$ by $q_t(f)(x) = f(x, t)$ is a homotopy equivalence, and thus induces an isomorphism 
\[K_*(C^*_r(\G \times [0,1], \omega)) \to K_*(C^*_r(\G, \omega_t)).\] 
\label{cpt-isom}
\end{proposition}

The proof will proceed through a series of lemmas, some of which have slightly more general hypotheses than those in the statement of Proposition \ref{cpt-isom}.

We begin with a definition.

\begin{definition} Let $\omega: \mathcal{G}\2 \to \T$ be a 2-cocycle on a groupoid $\G$. A function $\theta: \mathcal{G}\2 \to \R$ is a \emph{cocycle logarithm} for $\omega$ if
\begin{equation*}
\theta(x,yz) + \theta(y,z) = \theta(xy, z) + \theta(x,y)
\end{equation*}
and 
\[\omega(x,y)= \exp(\theta(x,y)) := \mathrm{e}^{2\pi \mathrm{i} \theta(x,y)}\]
for all $(x, y) \in G\2$.   
\end{definition}

\begin{lemma}
Let $\G$ be a compact groupoid, and let $\omega$ be a  homotopy of  cocycles on $\G$.  Given $t\in[0,1]$, we define a cocycle $h_t$ on $\G \times [0,1]$ by 
\[h_t(x, y, s) = \omega(x, y, t).\]  
For each $t \in [0,1]$, we can find $\epsilon_t > 0$ such that the cocycle $H_t$, given on\\ $\G \times \left((t-\epsilon_t, t+\epsilon_t) \cap [0,1]\right)$ by 
\[H_t(x,y,s) := \omega(x,y,s)h_t(x,y,s)\inv,\]
admits a continuous cocycle logarithm.  
\label{coh-triv}
\end{lemma}
\begin{proof}
Observe that $\omega$ and $h_t$ are continuous cocycles on $\G \times [0,1]$ by definition.  Moreover, Proposition 1.10 in \cite{geoff-thesis} tells us that $\G\2 \times [0,1]$ is  a closed subset of a compact space, and hence is compact.  Thus, for each $t \in [0,1]$ we can find $\epsilon_t >0$ such that $|s-t| < \epsilon_t$ implies that $H_t(x,y,s)$  lies in the right-hand half of the circle, strictly between $-i$ and $i$, for all $(x,y) \in \G\2$. This implies that  the principal branch of the logarithm is a cocycle logarithm for $H_t$, for any $t$: Let $\theta_{t}(x,y,s)$  be the argument of $H_t(x,y,s)$ that lies in the interval $(-\pi, \pi)$.  Then since the product of cocycles is a cocycle by Remark \ref{coh-gp}, the cocycle identity for $H_t$  implies that for each fixed $(x,y)\in \G\2$ and $s \in [0,1]$,
\[\theta_{t}(x,y,s) + \theta_{t}(xy, z,s) = \theta_{t}(x, yz, s) + \theta_{t}(y,z, s) + 2\pi k\]
for some $k \in \Z$. If $k =0$ then the principal branch of the logarithm is a cocycle logarithm as claimed. 
 
 In order to have $k \not= 0$ we must have that at least one of the pairs 
\[\{\theta_{t}(x,y,s), \theta_{t}(x, yz, s)\}, \{\theta_{t}(xy, z,s), \theta_{t}(y, z,s)\}\]
 has a difference of at least $\pi$.  However, since $-\pi/2 < \theta_{t}(x,y,s) < \pi/2$ for all $(x,y) \in \G\2$, this is impossible, so the principal branch of the logarithm is a cocycle logarithm as claimed. 
\end{proof}
 
When a cocycle $\omega$ admits a logarithm, we can use left invariance of the Haar system to show that in many cases, $\omega$ is cohomologous to the trivial cocycle. This will be an important ingredient in our proof, because cohomologous cocycles give rise to isomorphic $C^*$-algebras.

\begin{definition}
Let $\omega_0, \omega_1$ be two cocycles on a groupoid $\mathcal{G}$.  We say that $\omega_0, \omega_1$ are \emph{cohomologous} if there exists a function $b: \mathcal{G} \to \T$ such that for any $(x, y) \in \mathcal{G}\2$, we have 
\[\omega_0(x,y) = (\delta b) (x,y) \omega_1(x,y) := b(x) b(y) b(xy)\inv \omega_1(x,y).\]
\end{definition}

\begin{remark}
When $\G$ is a LCH groupoid and $\omega_0, \omega_1$ are continuous 2-cocycles that are cohomologous via a continuous function $b$, then a straightforward check shows  
that  the map $f \mapsto f \cdot b$ on $C_c(\mathcal{G}) $ induces an isomorphism\footnote{This isomorphism in fact holds for both the full and the reduced $C^*$-algebras.}
\[C^*_{r}(\G, \omega_0) \cong C^*_{r}(\G, \omega_1).\]
\end{remark}

The class of groupoids alluded to above, for which the existence of a cocycle logarithm for $\omega$ implies that $\omega$ is cohomologous to the trivial cocycle, is the class of \emph{proper} groupoids.

\begin{definition}
A groupoid $\G$ is \emph{proper} if the map $(r,s): \G \to \G\z \times \G\z$ is proper.
\end{definition}

\begin{remark}
Proposition 1.10 of \cite{geoff-thesis} tells us that $\G\z \subseteq \G$ is closed in any LCH groupoid.  Consequently, if $\G$ is compact then 
$\G$ is proper.
\end{remark}

\begin{example} A transformation group $G \ltimes X$ is proper iff $G$ acts properly on $X$ -- that is, iff for any $K_1, K_2 \subseteq X$ compact, 
$\{\gamma \in G: \gamma K_1 \cap K_2 \not= \emptyset\}$ is compact.  In particular, if $G$ is compact, then $G \ltimes X$ is always proper. 
\end{example}

If $\G$ is a proper groupoid, then Tu describes in Proposition 6.11 of \cite{tu-novikov} how to construct a continuous \emph{cutoff function} $c: \G\z \to \R^+$ such that for any $u \in \G\z$, 
\begin{equation}\int_{\G^u} c(s(x)) \mathrm{d}\lambda^u(x) = 1
\label{cutoff}
\end{equation} 
We will use this cutoff function to show that if a 2-cocycle $\omega$ on a proper groupoid admits a logarithm, then $\omega$ is cohomologous to the trivial cocycle.

\begin{proposition}
Let $\G$ be a proper groupoid.  Suppose that a 2-cocycle $\omega$ on $\G$ admits a  cocycle logarithm $\theta: \G\2 \to \R$.  Then $\omega$ is cohomologous to the trivial cocycle on $\G$.
\end{proposition}
\begin{proof} Define 
\begin{equation*}
b(x) = \int_{\G^{s(x)}} \theta(x,z) c(s(z)) \, \mathrm{d}\lambda^{s(x)}(z).
\end{equation*}
Using the left-invariance of the Haar system, and \eqref{cutoff}, we see that whenever $(x,y) \in \G\2$,
\begin{align*}
b(x) + b(y) - b(xy) &= \int_{\G^{s(x)}} \theta(x,z) c(s(z)) \, \mathrm{d}\lambda^{s(x)}(z) \\
&\qquad + \int_{\G^{s(y)}} \theta(y,z) c(s(z)) \, \mathrm{d}\lambda^{s(y)}(z)\\
&\qquad -  \int_{\G^{s(y)}}  \theta(xy, z) c(s(z)) \, \mathrm{d}\lambda^{s(y)}(z) \\
&= \int_{\G^{s(y)}} \left(\theta(x, yz) + \theta(y,z) - \theta(xy, z) \right) c(s(z)) \, \mathrm{d}\lambda^{s(y)}(z) \\
&= \int_{\G^{s(y)}} \theta(x,y) c(s(z)) \, \mathrm{d}\lambda^{s(y)}(z) \\
&= \theta(x,y).
\end{align*}
Thus,
\begin{align*}
\omega(x,y) &= \exp(\theta(x,y)) = \exp(b(x)) \, \exp(b(y)) \, \exp(b(xy))\inv \\
&= \delta (\exp\circ b)(x,y),
\end{align*}
so $\omega$ is a coboundary, as claimed. \end{proof}

\begin{remark}
We note that since $c$ is continuous, $b$ will be continuous whenever $\theta$ is.  Consequently, if a continuous 2-cocycle $\omega$ on a LCH groupoid $\G$ admits a continuous logarithm, then $C^*_r(\G, \omega) \cong C^*_r(\G)$.
\end{remark}

With these lemmas in hand, we can now proceed with the proof of Proposition~\ref{cpt-isom}.

\begin{proof}[Proof of Proposition~\ref{cpt-isom}]
The two previous Lemmas,
 combined with our \\ earlier observation that cohomologous cocycles induce isomorphic twisted $C^*$-algebras, imply that if $\omega$ is a homotopy of continuous cocycles on a compact transformation group $\G = H \ltimes X$, then for any $t \in [0,1]$, there exists $\epsilon_t > 0$ such that
\[C^*_r(\G \times (t-\epsilon_t, t+\epsilon_t), \omega) \cong C^*_r(\G \times (t-\epsilon_t, t+\epsilon_t), h_t).\] 

We claim that $C^*_r(\G \times (t - \epsilon_t, t + \epsilon_t), h_t) \cong C_0((t -\epsilon_t, t+\epsilon_t), C^*_r(\G, \omega_t)).$  
To see this, let $\varepsilon = (t-\epsilon_t, t+\epsilon_t)$ and observe that $C_c(\varepsilon \times \G)$ is dense in both algebras.  Since the cocycle $h_t$ doesn't depend on $s \in \varepsilon$, the multiplication and involution operations on $C_c(\varepsilon \times \G)$ are the same in both algebras.  

To see that the norms agree on the two algebras, recall that in the groupoid $\G \times \varepsilon$, we have $(x,s)\inv = (x\inv, s)$.  If we write $\nu\inv$ for the measure on $\G$ induced from a full measure $\mu$ on $\G\z$ as in Definition \ref{l2}, and $\tilde{\nu}\inv$ for the measure on $\G \times \varepsilon$ induced from $\mu$ and from Lebesgue measure on $\varepsilon$, then
\begin{align*}
\int_{\G \times \varepsilon} f(x,s) \mathrm{d}\tilde{\nu}\inv(x,s) &= \int_\varepsilon \int_{\G\z} \int_{\G^u} f(x\inv, s) \mathrm{d}\lambda^u(x) \, \mathrm{d}\mu(u) \, \mathrm{d}s\\
& = \int_\varepsilon \int_{\G} f(x, s) \mathrm{d}\nu\inv(x) \mathrm{d}s.\end{align*}

In other words, $L^2(\tilde{\nu}\inv) = L^2(\varepsilon, L^2(\nu\inv)) \cong L^2(\varepsilon) \otimes L^2(\nu\inv)$.  Moreover, a quick examination of the formula for the representation $\text{Ind}\, \mu$ of $C_c(G \times \varepsilon)$ on $L^2(\tilde{\nu}\inv) \cong L^2(\varepsilon) \otimes L^2(\nu\inv)$ will show that this representation decomposes into pointwise multiplication on $L^2(\varepsilon)$ and convolution on $L^2(\nu\inv)$.  Thus, the representation $\text{Ind}\, \mu$ of the groupoid convolution algebra $C_c(\G \times \varepsilon, h_t)$ is the same as the representation of $C_c(\varepsilon) \odot C_c(\G, \omega_t)$ on $L^2(\varepsilon) \otimes L^2(\nu\inv)$ by multiplication in the first component, and twisted convolution multiplication in the second.  Since this latter representation gives rise to $C_0(\varepsilon) \otimes C^*_r(\G, \omega_t) \cong C_0(\varepsilon, C^*_r(\G, \omega_t))$, we have proved that 
\begin{equation*}
C^*_r(\G \times (t - \epsilon_t, t + \epsilon_t), h_t) \cong C_0((t -\epsilon_t, t+\epsilon_t)) \otimes C^*_r(\G, \omega_t)
\end{equation*}
as claimed.

In sum, when $\G$ is a compact groupoid, $C^*_r(\G \times [0,1], \omega)$ is a locally trivial continuous field of $C^*$-algebras over $[0,1]$, with fiber algebra $C^*_r(\G, \omega_t)$ over $t \in [0,1]$.  Consequently, the fiber algebras $C^*_r(\G, \omega_t)$ are all  isomorphic. 
Since any locally trivial fiber bundle over a contractible space is trivializable, we have 
\[C^*_r(\G \times [0,1], \omega) \cong C([0,1], C^*_r(\G, \omega_t)) \cong C([0,1]) \otimes C^*_r(\G, \omega_t).\]
Moreover, since $[0,1]$ is compact and contractible, the natural map 
\[q_t:  C^*_r(\G \times [0,1], \omega) \to C^*_r(\G, \omega_t)\]
is a homotopy equivalence, and induces an isomorphism 
\[K_*(C^*_r(\G \times [0,1], \omega)) \cong K_*(C^*_r(\G, \omega_t))\]
as claimed. 
\end{proof}

\section{TWISTED CROSSED PRODUCTS}
\label{pr}
In this section, we connect our construction of the reduced twisted transfor-\\ mation-group $C^*$-algebra from Section \ref{2} with that of Packer and Raeburn in \cite{twisted-xprod1, twisted-xprod2}.  Our eventual goal is to invoke the Packer-Raeburn ``stabilization trick'' (Theorem 3.4 in \cite{twisted-xprod1}) to show that 
\[C^*_r(G \ltimes X, \omega) \otimes \K \cong C_0(X, \K) \rtimes_r G.\]
In order to do this we need to exhibit an isomorphism between the reduced twisted groupoid $C^*$-algebra $C^*_r(G \ltimes X, \omega)$ as defined in Section \ref{2}, and the reduced twisted crossed product $C^*$-algebra $\tilde{\pi} \times R(C_0(X) \rtimes_{\alpha,  {\bf u}}G)$ described in Definition 3.10 and Remark 3.12 of \cite{twisted-xprod1}.  Since Packer and Raeburn restrict their attention in \cite{twisted-xprod1} to separable $C^*$-algebras and locally compact second countable groups, we begin here to impose those hypotheses as well, which will be in force throughout the rest of this paper.

To that end, suppose that $G \ltimes X$ is a second countable, locally compact Hausdorff transformation group, and that $\omega$ is a continuous 2-cocycle on $G \ltimes X$.  Writing $A = C_0(X)$, the action of $G$ on $X$ that underlies the transformation-group structure gives rise to a continuous action $\alpha$ of $G$ on $A$ by 
\begin{equation}
\alpha_\gamma(f)(v) := f(\gamma\inv \cdot v).\label{pr-action}
\end{equation}
Moreover, if we define ${\bf u}: G \times G \to C(X,\T) = UM(A)$ by 
\[{\bf u}(\gamma, \eta)(v) := \omega((\gamma, v),(\eta, \eta\inv \cdot v)),\]
then the fact that $\omega$ satisfies the cocycle condition \eqref{cocycle} tells us that for any $\gamma,\eta,\rho \in G$ we have 
\begin{equation}
\alpha_\gamma({\bf u}(\eta,\rho)) {\bf u}(\gamma, \eta \rho) =  {\bf u}(\gamma, \eta){\bf u}(\gamma \eta, \rho).\label{pr-cocycle}
\end{equation}
Furthermore, the continuity of $\omega$ implies that {\bf u} is also continuous, and hence Borel.
In other words, $(A, G, \alpha, {\bf u})$ is a twisted dynamical system in the sense of \cite{twisted-xprod1} Definition 2.1.

\begin{definition}[\cite{twisted-xprod1} Definition 2.1]
Let $A$ be a separable $C^*$-algebra, $G$ a locally compact second countable group, and let $\alpha: G \to \Aut A, {\bf u}: G \times G \to UM(A)$ be Borel maps.  The quadruple $(A, G, \alpha, {\bf u})$ is a \emph{twisted dynamical system} if
\bi 
\item $\alpha_e = id$ and ${\bf u}(\gamma, e) = {\bf u}(e, \gamma) = 1$ 
\item $\alpha_\gamma \circ \alpha_\eta = \Ad(u(\gamma, \eta)) \circ \alpha_{\gamma \eta}$
\item $\alpha_\gamma ({\bf u}(\eta, \rho)) {\bf u}(\gamma, \eta \rho)  =  {\bf u}(\gamma, \eta){\bf u}(\gamma \eta, \rho)$
\ei
for all $\gamma, \rho, \eta \in G$.
\end{definition}

  Conversely, if $(A, G, \alpha, {\bf u})$ is any twisted dynamical system with $A = C_0(X)$ abelian, the action  $\alpha$ of $G$ on $A$ must arise from an action of $G$ on $X$.  Then if we define $\omega:(G \ltimes X)\2 \to \T$ via 
\[\omega((\gamma, v), (\eta, \gamma\inv \cdot v)) :={\bf u}(\gamma,\eta)(v),\]
an easy check will show that $\omega$ satisfies the cocycle condition  \eqref{cocycle}.

If we additionally assume that $\alpha, {\bf u}$ are continuous, then $\alpha$ gives rise to a locally compact Hausdorff transformation group $G \ltimes X$ and the cocycle $\omega$ associated to {\bf u} is continuous.
In other words, \emph{continuous} twisted dynamical systems $(A, G, \alpha, {\bf u})$ with $A= C_0(X)$ abelian are in bijection with continuous 2-cocycles $\omega$ on the transformation group $G \ltimes X$.  Our goal in this section is to check that the corresponding reduced $C^*$-algebras $\tilde{\pi} \times R( A \times_{\alpha, {\bf u}} G ), C^*_r(G \ltimes X, \omega)$ are isomorphic.

 The reduced twisted crossed product of a twisted dynamical system\\ ($A, G, \alpha, {\bf u})$ is defined in \cite{twisted-xprod1} Remark 3.12 as the image of the universal twisted crossed product $A \times_{\alpha, {\bf u}} G$ under the representation 
 \[\tilde{\pi} \times R: A \times_{\alpha, {\bf u}} G \to B(L^2(G, \H)\]
induced from a faithful representation $\pi: A \to B(\H)$. An explicit formula for the covariant representation $(\tilde{\pi}, R)$ associated to $\pi$ is given in \cite{twisted-xprod1} Definition 3.10. However, we can also view the reduced twisted crossed product as a completion of $C_c(G, A)$, because Definition 2.4 of \cite{twisted-xprod1} implies that $C_c(G, A)$ maps densely into $A \times_{\alpha, {\bf u}} G$.  This perspective will make it easier
  to compare $C^*_r(G \ltimes X, \omega)$ and $\tilde{\pi} \times R( A \times_{\alpha, {\bf u}} G)$. 
  
   Conditions $(b)$ and $(c)$ of Definition 2.4 of \cite{twisted-xprod1}, combined with Lemma 1.92 in \cite{xprod}, tell us that if $(\pi, U)$ is a covariant representation of $(A,G, \alpha, {\bf u})$ on the Hilbert space $\H$, then the integrated form $\pi \times U: A \times_{\alpha, {\bf u}} G \to B(\H)$ is given on $C_c(G, A)$ by 
 \[\pi \times U(f) = \int_G \pi(f(\gamma)) U_\gamma \mathrm{d}\lambda(\gamma).\]
Thus, using the formulas for $(\tilde{\pi}, R)$ given in Definition 3.10 of \cite{twisted-xprod1}, we see that the reduced twisted crossed product $\tilde{\pi} \times R( A \times_{\alpha, {\bf u}} G )$ is the completion of  $C_c(G, A)$ in the norm given by the representation $\tilde{\pi} \times R$, where (for $f \in C_c(G, A),\\  \xi \in L^2(G, \H)$)
\[\tilde{\pi} \times R(f)\xi(\gamma)= \int_G \pi(\alpha_\gamma( f(\eta)) \Delta(\eta)^{1/2} {\bf u}(\gamma, \eta))\xi(\gamma\eta) \, \mathrm{d}\lambda(\eta).\]
Here $\Delta: G \to \R^+$ denotes the modular function of $G$.

When we specialize to our case of interest (namely, continuous twisted dynamical systems with $A = C_0(X)$), using the faithful representation $\pi$ of $C_0(X)$ on $L^2(X, \mu)$ as multiplication operators for some full measure $\mu$ on $X$, the formula above becomes 
\begin{equation}
\tilde{\pi} \times R(f)\xi(\gamma, v)= \int_G f(\eta, \gamma\inv v) \Delta(\eta)^{1/2} \omega( (\gamma,v), (\eta, \gamma\inv v)) \xi(\gamma\eta, v) \, \mathrm{d}\lambda(\eta).
\label{red-xprod}
\end{equation}

We now check that there is a convolution structure on $C_c(G, C_0(X))$ such that the $*$-algebras $C_c(G \ltimes X, \omega)$ and $C_c(G, C_0(X))$ are isomorphic, and that the reduced norms gives rise to an isomorphism of $C^*$-algebras 
\[C^*_r(G \ltimes X, \omega) \cong\tilde{\pi} \times R( A \times_{\alpha, {\bf u}} G).\]
 While this result may be well-known to experts, we include a proof here because we have not found a satisfactory reference to it in the literature.
\begin{proposition}
Let $G \ltimes X$ be a second countable, locally compact Hausdorff transformation group, and let $\omega$ be a continuous 2-cocycle on $G \ltimes X$.  Fix a full measure $\mu$ on $X$.  There is an associated twisted dynamical system $(C_0(X), G, \alpha, {\bf u})$ such that $C^*_r(G \ltimes X, \omega)$ is isomorphic to the reduced twisted crossed product $\tilde{\pi} \times R(C_0(X) \rtimes_{\alpha,  {\bf u}}G)$ described in Remark 3.12 of \cite{twisted-xprod1}, where $\pi: C_0(X) \to B(L^2(X, \mu))$ is the representation of $C_0(X)$ on $L^2(X, \mu)$  by multiplication operators.
\label{pr&gpoid}
\end{proposition}
\begin{proof}
We begin by observing that $C_c(G, C_0(X))$ forms a subalgebra of\\ $\tilde{\pi} \times R(C_0(X) \rtimes_{\alpha,  {\bf u}}G)$, not merely a dense subspace:
if $f, g \in C_c(G, C_0(X))$ then 
 \[f * g(\gamma, v) := \int f(\eta,v) g(\eta\inv \gamma, \eta\inv\cdot v) \omega((\eta,v), (\eta\inv \gamma, \eta\inv \cdot v))\, \mathrm{d}\lambda(\gamma)\]
\[ f^*(\gamma,v) := \overline{f(\gamma\inv, \gamma\inv \cdot v) \Delta(\gamma\inv) \omega((\gamma\inv, \gamma\inv \cdot v), (\gamma,v))}\]
are also elements of $C_c(G, C_0(X))$, and one can check, using \eqref{red-xprod}, that the $*$-algebra structure thus defined is preserved by the representation $\tilde{\pi} \times R$.

Using this $*$-algebraic structure on $C_c(G, C_0(X))$, we now define a $*$-homo-\\ morphism $\phi: C_c(G \ltimes X, \omega) \to C_c(G, C_0(X))$ and a unitary 
\[U: L^2(G \ltimes X, \nu\inv) \to L^2(G, L^2(X)),\] that intertwine the representations $\tilde{\pi} \times R$ and $\text{Ind}\, \mu$.  That is,  for any $f \in C_c(G \ltimes X, \omega), \xi \in L^2(G, L^2(X))$, we will show that 
\begin{equation}
\tilde{\pi} \times R(\phi(f))(\xi) = U \, \text{Ind}\,\mu(f)(U^*\xi).
\label{intertwine}
\end{equation}
It follows that $\phi$ is norm-preserving, and therefore extends to a $*$-homomorphism $C^*_r(G\ltimes X, \omega) \to \tilde{\pi} \times R(C_0(X) \rtimes_{\alpha,  u}G)$ that implements the desired isomorphism.

The appearance of the modular function in the involution in $C_c(G, C_0(X)) \subseteq   \tilde{\pi} \times R(C_0(X) \rtimes_{\alpha,  u}G)$ above means that the standard inclusion $C_c(G \times X) \hookrightarrow C_c(G, C_0(X))$ will not be a $*$-homomorphism.  Instead, we define $\phi: C_c(G \ltimes X, \omega) \to C_c(G, C_0(X))$  by 
\[\phi(f)(\gamma, v) := f(\gamma, v) \Delta(\gamma\inv)^{1/2};\]
a straightforward check shows that $\phi$ is multiplicative and $*$-preserving.

Defining $U: L^2(G \ltimes X, \nu\inv) \to L^2(G, L^2(X))$ by 
\begin{align*}
U\xi(x) &=  \xi(x\inv) \omega(x, x\inv) \\
&= \xi(\gamma\inv, \gamma \inv \cdot v) \omega((\gamma, v), (\gamma\inv, \gamma\inv \cdot v)) 
\end{align*}
if $x = (\gamma, v) \in G \ltimes X$,
a similarly straightforward check (invoking the cocycle condition \eqref{cocycle}) shows that \eqref{intertwine} holds, and moreover that $U$ is a unitary operator.  It follows (as remarked above) that
\[C^*_r(G \ltimes X, \omega) \cong \tilde{\pi} \times R(C_0(X) \rtimes_{\alpha,  {\bf u}}G)\] as desired. 
\end{proof}

\section{THE MAIN THEOREM}
\label{6}

We are now ready to tackle the proof of Theorem \ref{main}.  Our proof closely parallels the proof given in the group case (Theorem 1.9 in \cite{ELPW}) by Echterhoff, L\"uck, Phillips, and Walters; thus, our first step is to find an element 
\[\textbf{x} \in KK^G(C_0(X \times [0,1],\K),C_0(X,\K))\]
 in the $G$-equivariant $KK$-theory of $A:= C_0(X \times [0,1], \K)$ and $B:= C_0(X,\K)$, such that if $H \leq G$ is compact, and we write $\textbf{x}^H$ for the element $\textbf{x}$ thought of as merely an $H$-equivariant $KK$-element, the map
\[\underline{ \quad} \# \textbf{x}^H: KK^H(\C, A) \to KK^H(\C, B)\]
given by taking the Kasparov product with $\textbf{x}^H$ is an isomorphism for any $H \leq G$ compact.  We will do this by showing that $\textbf{x}$ corresponds to the element 
\[[q_t]\in KK(C^*_r(G \ltimes X \times [0,1], \omega), C^*_r(G \ltimes X, \omega_t))\] arising from the $*$-homomorphism $q_t:  C^*_r(G \ltimes X \times [0,1], \omega) \to C^*_r(G \ltimes X, \omega_t)$ of ``evaluation at $t \in [0,1]$,'' since we know from Proposition \ref{cpt-isom} that $q_t$ induces a homotopy equivalence, and hence a $KK$-equivalence, when the transformation group $G \ltimes X$ is compact.  Thus, if $X$ is compact, $\underline{ \quad} \# \textbf{x}^H$ is an isomorphism for any compact subgroup $H \leq G$.  In other words, when $X$ is compact, $\textbf{x}$ satisfies the hypotheses of Proposition 1.6 of \cite{ELPW}, so it follows that $\textbf{x}$ (and thus $[q_t]$) gives rise to an isomorphism on $K$-theory, as claimed.

When $X$ is not compact, for each $H \leq G$ compact we can write $C_0(X, \K) = \varinjlim C_0(Y_{i,H}, \K)$ as an inductive limit, where $Y_{i,H}$ is a pre-compact $H$-invariant subspace of $X$ such that $\overline{Y_{i,H}}$ is also $H$-invariant.  We then use the $K$-theoretic 6-term exact sequence, the continuity of $K$-theory, and   the result established in the case of a compact base space to see that $[q_t]$ also gives rise to an isomorphism 
\begin{equation*}
[q_t]: K_*(C^*_r( G \ltimes X \times [0,1], \omega)) \to K_*(C^*_r(G \ltimes X, \omega_t))
\end{equation*}
 when $X$ is not compact.  (We thank the anonymous referee for suggesting this line of argument to us.)

\subsection{$KK$-theory}
We begin by reviewing a few fundamental constructions in equivariant $KK$-theory.  
First, recall that if two separable $C^*$-algebras $A, B$ admit actions by a second countable locally compact Hausdorff group $G$, then an element of $KK^G(A,B)$ is given by an equivalence class of $KK$-triples $(\mathcal{E}, T, F)$, where $\mathcal{E}$ is a $\Z/2\Z$-graded right Hilbert $B$-module admitting an action of $G$; $T: A \to L(\mathcal{E})$ is a graded $*$-homomorphism from $A$ into the bounded adjointable operators on $\mathcal{E}$; $F \in L(\mathcal{E})$ is a degree-1 operator; and the module operations on $\mathcal{E}$ and the operators $T,F$ are all $G$-equivariant.

In particular, \cite{blackadar} Examples 17.1.2(a) tells us that if $\psi: A \to B$ is  a $G$-equivariant $*$-homomorphism, then $\psi$ gives rise to a $KK$-triple $(B, \psi, 0)$ and hence to an element $[(B, \psi, 0)] \in KK^G(A, B)$. Almost all of the $KK$-elements that will concern us in this paper will be of this form. We will often denote $[(B, \psi, 0)]$ simply by $[\psi]$.

Given a $G$-algebra $A$, we can also use $KK$-theory to define  the \emph{topological $K$-theory of $G$ with coefficients in $A$} as the abelian group
\[K^{top}_*(G; A) = \lim_{L \subseteq \mathcal{E}G} KK^G_*(C_0(L), A),\]
where the limit is taken over all $G$-compact subspaces $L$ of a universal proper $G$-space $\mathcal{E}G$.  There is a natural map (see \cite{baum-connes-higson} Section 9) from the topological $K$-theory of $G$ with coefficients in $A$ to the usual $K$-theory group of the reduced crossed product $A \rtimes_r G$:
\[\mu: K^{top}_*(G; A) \to K_*(A \rtimes_r G).\]
The map $\mu$ is called the \emph{assembly map}, and $G$ is said to satisfy the \emph{Baum-Connes conjecture with coefficients} if $\mu$ is an isomorphism for any $G$-algebra $A$.  As mentioned in Remark \ref{b-c}, every a-T-menable group satisfies the Baum-Connes conjecture with coefficients.

Following \cite{xprod} Sections 2.2 and 7.2, 
 given a $G$-algebra $A$  and a faithful representation $\pi$ of $A$ on a Hilbert space $\H$, we define the \emph{reduced crossed product $A \rtimes_r G$} as the completion of the convolution algebra $C_c(G, A)$ in the norm coming from the  representation $\text{Ind}_e^G \pi$ of $C_c(G, A)$ on $L^2(G) \otimes \H$ induced from $\pi$.  Lemma 7.8
 in \cite{xprod} tells us that the reduced crossed product does not depend on the choice of faithful representation $\pi$.
 
 We will also have occasion to consider the \emph{full crossed product} $A \rtimes G$, which is defined in Lemma 2.27 of \cite{xprod}
 as a universal object for covariant representations of the dynamical system $(A, G)$.  Being a universal object, the full crossed product behaves well with respect to functorial constructions such as inductive limits; moreover, for amenable groups $G$, $A \rtimes G \cong A \rtimes_r G$.  We will take particular advantage of this ability to use the full and reduced crossed products interchangeably in the case when $G$ is compact.

For $A, B$ as above, we write $j^G: KK^G(A, B) \to KK(A \rtimes_r G, B \rtimes_r G)$ for the \emph{descent map} defined in \cite{kasparov-equivKK} Theorem 3.11. 
The descent homomorphism is functorial, and generalizes to $KK$-theory the natural map
\[\Hom_G(A,B) \to \Hom(A \rtimes_r G, B \rtimes_r G)\] 
in the category of $C^*$-algebras and (equivariant) $*$-homomorphisms.  
In other words, if $\phi: A \to B$ is a $*$-homomorphism, then $j^G([\phi]) = [\phi^G]$, where $\phi^G: A \rtimes_r G \to B \rtimes_r G$ is the $*$-homomorphism given on the dense $*$-subalgebra $C_c(G, A)$ by $\phi^G(f)(\gamma) = \phi(f(\gamma))$.

The deepest and most useful aspect of $KK$-theory is the \emph{Kasparov product} $\#$, which is a functorial map 
\[\#: KK^G_i(A,B) \times KK^G_j(B, D) \to KK^G_{i+j}(A,D),\]
where $i,j \in \{ 0,1\}$.
See \cite{kasparov-equivKK} Theorem 2.11 for more details.

Finally, Theorem 5.4 in \cite{kasparov-skandalis} tells us how to construct an element $\Lambda_{pt} \in KK_*(\C, C^*_r(G))$ such that if $G$ is compact and $B$ is a $C^*$-algebra with a $G$-action, then the map $KS: KK^G_*(\C, B) \to KK_*(\C, B \rtimes_r G)$ given by 
\[KS(x) = \Lambda_{pt} \# j^G(x)\]
is an isomorphism. 

\subsection{Proof of Theorem \ref{main}} 
Having reviewed the preliminaries, we now begin the process of finding the  element $\textbf{x} \in KK^G(A,B)$ alluded to at the beginning of this section, where $A = C_0(X \times [0,1], \K), \ B = C_0(X, \K)$ for a  $G$-space $X$.

Let $\omega$ be a homotopy of cocycles on $G\ltimes X$, and write $ev_t: A \to B$ for the $*$-homomorphism  $ev_t(f)(x) = f(x,t)$.  Proposition \ref{pr&gpoid} tells us how to construct  twisted dynamical systems 
\[(C_0(X \times [0,1]), G, \alpha, {\bf u}), (C_0(X), G, \alpha, {\bf u}_t)\] such that 
\begin{align*}
C^*_r(G \ltimes X \times [0,1], \omega) &\cong \tilde{\pi} \times R(C_0(X \times [0,1]) \rtimes_{\alpha,  {\bf u}}G);\\
C^*_r(G \ltimes X, \omega_t) &\cong \tilde{\pi} \times R(C_0(X) \rtimes_{\alpha, {\bf u}_t}G).
\end{align*}

Now, Theorems 3.4 and 3.11 of \cite{twisted-xprod1}
tell us how to construct, from the above twisted dynamical systems, actions $\beta, \beta_t$ of $G$ on $A, B$ respectively such that 
\begin{align*}
\tilde{\pi} \times R(C_0(X \times [0,1]) \rtimes_{\alpha,  {\bf u}}G) \otimes \K &\cong A \rtimes_{\beta,r} G, \\
\tilde{\pi} \times R(C_0(X) \rtimes_{\alpha,  {\bf u}_t}G) \otimes \K &\cong B \rtimes_{\beta_t, r} G.
\end{align*}
Moreover, examining the formula for $\beta$ given by Equation 3.1 in \cite{twisted-xprod1} in this particular case reveals immediately that $\beta$ 
preserves the fibers over $[0,1]$.  In other words, $ev_t$ is equivariant, 
 and consequently  induces a $*$-homomorphism 
 \[ev_t^G: A \rtimes_{r, \beta} G \to B \rtimes_{r, \beta_t} G,\] which is given on the dense $*$-subalgebra $C_c(G \times X \times [0,1],\K)$ by
\[ev_t^G(f)(\gamma, u) := f(\gamma, u, t).\]

Since the descent map $j^G: KK^G(A,B) \to KK(A \rtimes_r G, B \rtimes_r G)$ generalizes the map $\Hom_G(A,B) \to \Hom (A\rtimes_r G, B \rtimes_r G)$, it follows that 
\[j^G([ev_t]) = [ev_t^G].\]

We can now complete the proof of our main theorem.
\begin{theorem} 
\label{main}
Let $G \ltimes X$ be a second countable locally compact Hausdorff transformation group  such that the group $G$ satisfies the Baum-Connes conjecture with coefficients, and let $\omega$ be a homotopy of continuous cocycles on $G \ltimes X$.  For any $t \in [0,1]$,  the $*$-homomorphism \[q_t: C^*_r(G \ltimes X \times [0,1], \omega) \to C^*_r(G \ltimes X, \omega_t),\]
given on $C_c(G \ltimes X \times [0,1])$ by evaluation at $t \in [0,1]$, induces an isomorphism 
\[K_*(C^*_r(G \ltimes X \times [0,1], \omega)) \cong K_*(C^*_r(G \ltimes X, \omega_t)).\] 
\end{theorem}
\begin{proof}
We begin by examining the case when $H \leq G$ is a compact subgroup. 
 We will show that the diagram
\begin{equation}
\xymatrix{
 KK^H_*(\C, A)  \ar[r]^{\# [ev_t]} \ar[d]^{KS} & KK^H_*(\C, B) \ar[d]^{KS}\\
KK_*(\C, A \rtimes_{\beta,r} H) \ar[d]^{\Phi} \ar[r]^{\# [ev_t^H]} & KK_*(\C, B \rtimes_{\beta_t,r}H) \ar[d]^{\Phi_t}\\
KK_*(\C, C^*_r(H \ltimes X \times [0,1], \omega))  \ar[r]^{\# [q_t]} & KK_*(\C, C^*_r(H \ltimes X, \omega_t)) 
}
\label{diagram}
\end{equation}
commutes, where $A = C_0(X \times [0,1], \K), B = C_0(X, \K)$ as above. 
Here 
 $\Phi, \Phi_t$ denote the $KK$-equivalences induced by the $C^*$-algebraic isomorphisms 
\begin{align}
\phi: A \rtimes_{\beta,r} H & \cong C^*_r(H \ltimes X \times [0,1], \omega) \otimes \K,\\
\phi_t: B \rtimes_{\beta_t,r} H & \cong C^*_r(H \ltimes X, \omega_t) \otimes \K \label{phi}
\end{align}
provided by Theorems 3.4 and 3.11 of \cite{twisted-xprod1}
and Proposition \ref{pr&gpoid} of the present paper. 


The functoriality of the descent map $j^H$ 
tells us that if $\textbf{x} \in KK^H_*(\C, A)$, 
we have\footnote{This equation holds regardless of whether $H$ is  compact.} 
\[ j^H(\textbf{x}) \# j^H([ev_t]) = j^H(\textbf{x} \# [ev_t]) \in KK_*(C^*_r(H), B\rtimes_{\beta_t, r} H);\]
consequently, 
\begin{align*}
KS(\textbf{x} \# [ev_t]) & = \Lambda_{pt} \# j^H(\textbf{x} \# [ev_t]) = \Lambda_{pt} \# j^H(\textbf{x}) \# j^H([ev_t]) \\
& = \Lambda_{pt} \# j^H(\textbf{x}) \# [ev_t^H] \\
& = KS(\textbf{x}) \# [ev_t^H].
\end{align*}
In other words, the top square of the diagram commutes.

To see that the bottom square of the diagram commutes, since all of the maps in question come from $*$-isomorphisms, it suffices to check that 
\begin{equation}
(q_t \otimes \text{id}) \circ \phi = \phi_t \circ ev_t^H .\label{commute}
\end{equation}
Examining the formula for the actions $\beta, \beta_t$ given in Equation 3.1 of \cite{twisted-xprod1}, and the isomorphisms of our Proposition \ref{pr&gpoid} and Corollary 3.7 in \cite{twisted-xprod1}, we see that these all preserve the fiber over $t \in [0,1]$.  
Consequently, we have the desired equality \eqref{commute}, and the bottom square of \eqref{diagram} commutes as well.


We now specialize further, to the case when $X$ is compact.  In this case, the transformation groups $H \ltimes X \times [0,1], H \ltimes X$ are compact, so we know by Proposition \ref{cpt-isom} that $q_t: C^*_r(H \ltimes X \times [0,1], \omega) \to C^*_r(H \ltimes X, \omega_t)$ is a homotopy equivalence, and hence induces an isomorphism on $KK$-theory.  It follows from the commutativity of \eqref{diagram} that taking the Kasparov product with $[ev_t]$ induces an isomorphism  
\[KK^H_*(\C, A) \cong KK_*^H(\C, B)\]
 whenever $H$ and $X$ are compact.
In other words, the element  $[ev_t] \in KK^G_*(A,B)$ satisfies the hypotheses of Proposition 1.6 in \cite{ELPW}.  The conclusion of that proposition tells us that $ev_t$ induces an isomorphism 
\[K^{top}_*(G; A) \to K^{top}_*(G; B),\]
and since we hypothesized that $G$ satisfies the Baum-Connes conjecture with coefficients, we moreover have 
\[K_*(A \rtimes_{\beta, r} G) \cong K_*(B\rtimes_{\beta_t, r} G).\]
The isomorphism of \eqref{phi} now tells us that for any $t \in [0,1]$, we have 
\[K_*(C^*_r(G \ltimes X \times [0,1], \omega)) \cong K_*(C^*_r(G \ltimes X, \omega_t)).\]
It remains to check that this isomorphism is implemented by $q_t$, as claimed.

As described in Section 9 of \cite{baum-connes-higson}, the Baum-Connes assembly map $\mu$ is a modification of the descent map $j^G$.  Since we know that $j^G(ev_t) = ev_t^G$, it follows that the isomorphism induced by $ev_t$ on 
\[K_*(C^*_r(G \ltimes X \times [0,1], \omega)) \cong K_*(C^*_r(G \ltimes X, \omega_t)),\]
which we obtain by composing  $\mu$ with the $K$-theoretic isomorphism $\Phi_t$ induced by the $*$-isomorphism $\phi_t$ of \eqref{phi},
is also given by evaluation at $t$ on $C_c(G \ltimes X \times [0,1]) \subseteq C^*_r(G \ltimes X \times [0,1], \omega)$.
 This finishes the proof of Theorem \ref{main} in the case when $X$ is compact.

We now consider the case when $X$ is merely locally compact.  (We thank the referee for suggesting the following line of argument, and Siegfried Echterhoff for pointing out a flaw in our original implementation of it.)   Let $\omega$ be a homotopy of 2-cocycles on $G \ltimes X$, and let $H \leq G$ be compact.  
We can write 
\begin{equation*}
C_0(X \times [0,1], \K) = \varinjlim C_0(Y_i \times [0,1], \K),
\end{equation*}
where each $Y_i \times [0,1] \subseteq X \times [0,1]$ is an open pre-compact $H$-invariant subspace.  

Recall that the action $\beta$ of $H$ on $C_0(X \times [0,1], \K)$ which is associated to the 2-cocycle $\omega$ preserves the fiber over $t \in [0,1]$, so each subspace $Y_i \subseteq X$ is invariant under each action $\beta_t$ of $H$.
Moreover, (full) crossed products commute with inductive limits, so the isomorphisms
\begin{align*}
C^*_r(H \ltimes Y_i \times [0,1], \omega) \otimes \K & \cong C_0(Y_i \times [0,1], \K) \rtimes_{\beta, r} H \\ 
C^*_r(H \ltimes Y_i, \omega_t) \otimes \K & \cong C_0(Y_i, \K) \rtimes_{\beta_t,r} H 
\end{align*} 
of equation \eqref{phi} combine with the fact that for compact groups $H$, the full and reduced crossed products are isomorphic, to give us isomorphisms
\begin{align}
C_0(X \times [0,1], \K) \rtimes_{\beta} H &= \varinjlim C_0(Y_i \times [0,1], \K) \rtimes_\beta H \\
& = \varinjlim C^*_r(H \ltimes Y_i \times [0,1], \omega) \otimes \K, \\ 
C_0(X, \K) \rtimes_{\beta_t} H & = \varinjlim C_0(Y_i, \K) \rtimes_{\beta_t} H = \varinjlim C^*_r(H \ltimes Y_i, \omega_t) \otimes \K.
\label{ind-lim}
\end{align}

We also note that the compactness of $H$ allows us to assume that  $\overline{Y_i}$, and therefore  $\partial Y_i$, are also $H$-invariant.  Thus, for each $i$, we have a short exact sequence of $H$-algebras 
\[0 \to C_0(Y_i \times [0,1], \K) \to C_0(\overline{Y_i} \times [0,1], \K) \to C_0(\partial Y_i \times [0,1], \K) \to 0,\]
which translates into the short exact sequence 
\begin{equation}
\xymatrix{
0  \ar[r] &  C^*_r(H \ltimes Y_i \times [0,1], \omega) \otimes \K \ar[r] &  C^*_r(H \ltimes \overline{Y_i} \times [0,1], \omega) \otimes \K \ar[dl] \\
 &  C^*_r(H \ltimes \partial Y_i \times [0,1], \omega) \otimes \K \ar[r] & 0 }
\end{equation}
thanks to the isomorphism of \eqref{phi}, the fact that the (full) crossed product preserves exact sequences, and the fact that the full and reduced crossed product are isomorphic for a compact group such as $H$.  This short exact sequence gives us a 6-term exact sequence in $K$-theory:
\begin{equation}
 \begin{tikzpicture}
\node (F-0) at (-3,-1){$K_0(C^*_r(H \ltimes Y_i \times [0,1], \omega))$};
\node (F1+F2-0) at (-3,0){$K_0(C^*_r(H \ltimes \overline{Y_i} \times [0,1], \omega))$};
\node (F1F2-0) at (-3, 1){$K_0( C^*_r(H \ltimes \partial Y_i \times [0,1], \omega))$};
\node (F-1) at (3, 1){$K_1(C^*_r(H \ltimes Y_i \times [0,1], \omega) )$};
\node(F1+F2-1) at (3,0){$K_1(C^*_r(H \ltimes \overline{Y_i} \times [0,1], \omega))$};
\node(F1F2-1) at (3,-1){$K_1( C^*_r(H \ltimes \partial Y_i \times [0,1], \omega))$};
\begin{scope}[->,  line width=0.7]
			\draw  (F-0) to  (F1+F2-0); 
			\draw (F1+F2-0) to  (F1F2-0);
			\draw (F1F2-0) to (F-1);
			\draw (F-1) to (F1+F2-1);
			\draw (F1+F2-1) to (F1F2-1);
			\draw (F1F2-1) to (F-0);
\end{scope}
\end{tikzpicture} 
\label{kthy-ind-lim}
\end{equation}
Note that we also have an analogous short exact sequence
\[0 \to C^*_r(H \ltimes Y_i, \omega_t) \otimes \K \to C^*_r(H \ltimes \overline{Y_i}, \omega_t) \otimes \K \to C^*_r(H \ltimes \partial Y_i, \omega_t) \otimes \K,\]
and thus an analogous 6-term exact sequence in $K$-theory, for each $t \in [0,1]$.

Since $H \ltimes \overline{Y_i}$ and $H \ltimes \partial Y_i$ are  compact transformation groups,   Proposition \ref{cpt-isom} implies  that for each $i$, the evaluation maps 
\begin{align*}
q_t: C^*_r(H \ltimes \overline{Y_i}  \times [0,1], \omega) & \to C^*_r(H \ltimes \overline{Y_i},\omega_t)\\
q_t: C^*_r(H \ltimes \partial Y_i \times[0,1], \omega) & \to C^*_r(H \ltimes \partial Y_i, \omega_t)
\end{align*}
both induce an isomorphism on $K$-theory:
\begin{align*}
 K_*(C^*_r(H \ltimes \overline{Y_i} \times [0,1], \omega)) &\cong K_*(C^*_r(H \ltimes \overline{Y_i}, \omega_t)), \\
 K_*(C^*_r(H \ltimes \partial Y_i \times [0,1], \omega)) & \cong K_*(C^*_r(H \ltimes \partial Y_i, \omega_t)) .
 \end{align*}
 Combined with the $K$-theory exact sequence \eqref{kthy-ind-lim}, these isomorphisms imply that $ q_t: C^*_r(H \ltimes Y_i \times [0,1], \omega) \to C^*_r(H \ltimes Y_i, \omega_t)$ also induces an isomorphism 
 \begin{equation*}
K_*(C^*_r(H \ltimes  Y_i \times [0,1], \omega))  \cong K_*(C^*_r(H \ltimes Y_i, \omega_t))
 \end{equation*}
 for each $i$.

The continuity of $K$-theory and the isomorphisms of \eqref{ind-lim} now translate $q_t$ into an isomorphism 
\[K_*(C_0(X\times [0,1], \K) \rtimes_\beta H) \cong  K_*(C_0(X, \K) \rtimes_{\beta_t} H) ,\]
and the commutativity of the diagram \eqref{diagram} takes this isomorphism to the isomorphism $\underline{\quad} \# [ev_t^H]: KK^H_*(\C, C_0(X \times [0,1], \K)) \to KK^H_*(\C, C_0(X, \K))$.

It follows that the hypotheses of Proposition 1.6 in \cite{ELPW} are again satisfied by the element $[ev_t^G] \in KK^G(C_0(X \times [0,1], \K), C_0(X, \K))$, and so we have an isomorphism 
\[K_*(C^*_r(G \ltimes X \times [0,1], \omega)) \cong K_*(C^*_r(G \ltimes X, \omega_t))\]
for any $t \in [0,1]$. The same arguments we employed above in the case where $X$ is compact also tell us that this isomorphism is induced by the $*$-homomorphism of evaluation at $t$.  This finishes the proof of Theorem \ref{main}.
\end{proof}

\begin{corollary}
Let $G \ltimes X$ be a second countable locally compact transformation group such that $G$ satisfies the Baum-Connes conjecture with coefficients, and let $\omega = \{\omega_t\}_{t \in [0,1]}$ be a homotopy of 2-cocycles on $G \ltimes X$.  The homotopy induces an isomorphism 
\[K_*(C^*_r(G \ltimes X, \omega_0)) \cong K_*(C^*_r(G \ltimes X, \omega_1))\]
of the $K$-theory groups of the reduced twisted transformation-group $C^*$-algebras.
\end{corollary}

 \section{FUTURE WORK}
The results established in this article provide an answer to the question ``When does a homotopy $\{\omega_t\}_{t \in [0,1]}$ of 2-cocycles on a groupoid $\G$ give rise to an isomorphism of the $K$-theory groups 
\[K_*(C^*_r(\G, \omega_0)) \cong K_*(C^*_r(\G, \omega_1))\]
of the twisted groupoid $C^*$-algebras?'' in the case when $\G = G \ltimes X$ is a transformation group.  
In order to apply the proof techniques we have employed here, following \cite{ELPW}, to a broader class of groupoids, we will need to work harder. Proposition 1.6 in \cite{ELPW}, which is the crucial technical lemma for Theorem 1.9 in \cite{ELPW} as well as for our Theorem \ref{main}, allows us to study a homotopy of cocycles $\omega$ on $G \ltimes X$ by simply studying the restriction of $\omega$ to $H \ltimes X$ for all compact subgroups $H$ of $G$.  Since compact groups are very friendly, well-understood objects, this latter question is much easier to solve.  As of this writing, we are unaware of any analogous simplification results for more general groupoids.

As mentioned in the Introduction, Kumjian, Pask, and Sims have used completely different techniques in \cite{kps-kthy} to show that if a cocycle $\omega$ on a higher-rank graph $\Lambda$ admits a cocycle logarithm, so that $\omega(\lambda, \mu)$ can be written as $\exp(c(\lambda, \mu))$ for some $\R$-valued 2-cocycle $c$ on $\Lambda$, then 
\[K_*(C^*(\Lambda, \omega)) \cong K_*(C^*(\Lambda)).\]
In a forthcoming paper \cite{eag-kgraph}, we extend the techniques of \cite{kps-kthy} to show that any homotopy of cocycles on $\Lambda$ gives rise to $K$-equivalent twisted higher-rank graph $C^*$-algebras; it has been suggested to us that these techniques might also apply to the case of Deaconu-Renault groupoids.

\begin{acknowledgements} 
This research formed part of my PhD thesis, and I would like to thank my advisor Erik van Erp for his encouragement and assistance.  I would also especially like to thank Jody Trout for his help and support throughout my PhD work. We also thank the anonymous referee for many helpful comments and suggestions, in particular for suggesting the line of argument used to extend the proof of Theorem \ref{main} from the case where $X$ is compact to the general case.  Siegfried Echterhoff also provided a helpful suggestion in this direction.

Much of the research in the present paper was completed during a visit to the Universidad Carlos III in Madrid, Spain; thanks are also due to the mathematics department there, in particular Fernando Lled\'o and Francisco Marcell\'an, for making that very enjoyable visit possible.
 \end{acknowledgements}

\end{document}